%% file: pessimistic-w-wo-coupling.tex
\documentclass[a4paper,american,reqno]{amsart}

\newif\ifSubmission \Submissionfalse

\input{setup}

\bibliography{pessimistic-w-wo-coupling}

\begin{document}

\title[On Coupling Constraints in Pessimistic Linear Bilevel
Optimization]{On Coupling Constraints in\\Pessimistic Linear Bilevel
  Optimization}

\author[D. Henke, H. Lefebvre, M. Schmidt, J. Thürauf]%
{Dorothee Henke, Henri Lefebvre, Martin Schmidt, Johannes Thürauf}

\address[D. Henke]{%
  University of Passau,
  School of Business, Economics and Information Systems,
  Chair of Business Decisions and Data Science,
  Dr.-Hans-Kapfinger-Str.\ 30,
  94032 Passau,
  Germany}
\email{dorothee.henke@uni-passau.de}

\address[H. Lefebvre, M. Schmidt]{%
  Trier University,
  Department of Mathematics,
  Universitäts\-ring 15,
  54296 Trier,
  Germany}
\email{henri.lefebvre@uni-trier.de}
\email{martin.schmidt@uni-trier.de}

\address[J. Thürauf]{%
  University of Technology Nuremberg (UTN),
  Department Liberal Arts and Social Sciences,
  Discrete Optimization Lab,
  Dr.-Luise-Herzberg-Str.\ 4,
  90461 Nuremberg,
  Germany}
\email{johannes.thuerauf@utn.de}

\date{\today}

\begin{abstract}
  \input{abstract}

\end{abstract}

\keywords{\input{keywords}}
\subjclass[2020]{\input{msc2020}}

\maketitle

\input{introduction}

\input{problem-statement}

\input{removing-ccs}

\input{conclusion}

\input{acknowledgement}

\printbibliography

\end{document}

%% file: setup.tex
\usepackage{babel}
\usepackage[utf8]{inputenc}
\usepackage[T1]{fontenc}
\usepackage{xspace}
\usepackage{todonotes}
\usepackage{accents}
\usepackage{csquotes}
\usepackage{microtype}
\usepackage[style = numeric-comp,
            maxbibnames = 100,
            maxcitenames = 2,
            giveninits = true,
            isbn = false,
            url = true,
            backend = bibtex]{biblatex}
\usepackage[colorlinks,
            citecolor=blue,
            urlcolor=blue,
            linkcolor=blue]{hyperref} %

\DeclareFieldFormat{emph}{#1}
\DeclareFieldFormat{italic}{#1}
\DeclareFieldFormat{title}{#1}
\DeclareFieldFormat{booktitle}{#1}
\DeclareFieldFormat{maintitle}{#1}
\DeclareFieldFormat{journaltitle}{#1}
\DeclareFieldFormat{subtitle}{#1}
\DeclareFieldFormat{citetitle}{#1}
\DeclareFieldFormat{eventtitle}{#1}
\DeclareFieldFormat[article]{volume}{#1}
\DeclareFieldFormat{series}{#1}

\makeatletter
\patchcmd{\@settitle}{\uppercasenonmath\@title}{\scshape\large}{}{}
\patchcmd{\@setauthors}{\MakeUppercase}{\scshape\normalsize}{}{}
\makeatother

\vfuzz \hfuzz
\newtheorem{theorem}{Theorem}[section]
\newtheorem{lemma}[theorem]{Lemma}
\newtheorem{proposition}[theorem]{Proposition}
\newtheorem{corollary}[theorem]{Corollary}
\theoremstyle{definition}

\newtheorem*{standingassumption}{Standing Assumption}

\input{macros}

%% file: macros.tex
\newcommand{\st}{\text{s.t.}}

\newcommand{\field}{\mathbb}

\newcommand{\reals}{\field{R}}

\newcommand{\R}{\reals}

\newcommand{\abbr}[1][abbrev]{#1.\xspace}

\newcommand{\eg}{\abbr[e.g]}

\newcommand{\define}{\mathrel{{\mathop:}{=}}}

\newcommand{\defset}[3][\defsep]{\set{#2#1#3}}
\newcommand{\Defset}[3][\defsep]{\Set{#2#1#3}}
\newcommand{\set}[1]{\{#1\}}
\newcommand{\Set}[1]{\left\{#1\right\}}

\DeclareMathOperator*{\argmin}{arg\,min}

\DeclareMathOperator{\proj}{proj}

\newcommand{\rev}[1]{#1}

%% file: abstract.tex
The literature on pessimistic linear bilevel optimization with coupling
constraints is rather scarce and it has been common sense that these
problems are harder to tackle than pessimistic bilevel problems without
coupling constraints.
In this note, we show that this is not the case.
To this end, given a pessimistic problem with coupling constraints, we
derive a pessimistic problem without coupling constraints that has the
same set of globally optimal solutions.
Moreover, our results also show that one can equivalently replace a
pessimistic problem with such constraints with an optimistic problem
without coupling constraints.
This paves the way of both transferring theory and solution techniques
from any type of these problems to any other one.

\ifSubmission
\begin{center}
  \vspace*{1em}
  \emph{Communicated by Anil Aswani.}
\end{center}
\fi

%% file: keywords.tex
Bilevel optimization,
Pessimistic bilevel optimization,
Optimistic bilevel optimization,
Coupling constraints%

%% file: msc2020.tex
90Cxx, %
91A65%

%% file: introduction.tex
\section{Introduction}
\label{sec:introduction}

Bilevel optimization gained significant scientific attention over the
last years and decades.
One of the main reasons is that this framework allows to model hierarchical
decision-making processes in which two agents interact.
The leader acts first while anticipating the optimal reaction of
the follower, who acts second and who takes the leader's decision
into account.
For introductions to the field, we refer to the books and lecture notes
by \textcite{Dempe:2002,Dempe-et-al:2015,Beck_Schmidt:2019}, in which
the interested reader can also find many illustrative examples.
This field of study dates back to the seminal contributions by
\textcite{Stackelberg:1934,Stackelberg:1952}, while the mathematical
optimization and operations research communities started to
investigate these problems in the 1970s and 1980s; see, e.g.,
\textcite{Bracken-Mcgill:1973,Candler-Norton:1977,Bialas-Karwan:1984}
for the earliest publications.
Since then, many scientific advances have been achieved ranging from
theoretical studies on, e.g., existence of solutions or optimality
conditions over structural insights and reformulations to algorithmic
approaches for actually solving these challenging problems.

In this note, we consider two key aspects of the field of bilevel
optimization.
First, bilevel problems are generally ill-posed in case of
multiplicities in the set of optimal reactions of the follower.
Usually, this is resolved by fixing the level of cooperation of the
follower, i.e., choosing a follower's solution that is in favor of
the leader's objective or a solution that is the worst for the
leader.
The former concept leads to the so-called optimistic bilevel problem
whereas the latter is referred to as the pessimistic bilevel problem;
see, e.g., \textcite{Dempe:2002} for a general introduction.
Second, many contributions in bilevel optimization consider either the
case with coupling constraints or without them, where a coupling
constraint is an upper-level constraint that explicitly depends on the
variables of the follower.

\rev{Regarding the first aspect, it was common sense in the literature
  that the pessimistic bilevel problem is harder to tackle---both in
  theory and practice.
  Hence, it got significantly less attention compared to optimistic
  bilevel optimization.
  Nevertheless, important contributions such as regarding optimality
  conditions \parencite{Dempe-et-al:2014} or
  reformulations \parencite{Aussel2019} have been made in the last
  years.
  Actually, very recently, it has been shown by \textcite{Zeng:2020}
  that pessimistic bilevel optimization problems are not harder to
  tackle in the sense that, for every pessimistic bilevel problem, one can
  derive an optimistic problem that can be solved instead.}

Regarding the aspect of coupling constraints, it has been known for
more than 25~years that they may lead to disconnected bilevel feasible
sets \parencite{Henke-et-al:2024}, while it is shown by
\textcite{Benson:1989} that the feasible set of
a linear bilevel problem without coupling constraints is always
connected.
The possibility of disconnected feasible sets mainly gained
prominence because it allows to model mixed-binary linear problems using
purely continuous linear bilevel models, see, \eg, Section~3
in~\textcite{Vicente-et-al:1996} and Section~3.1
in~\textcite{Audet-et-al:1997}.
This also played an important role
in early NP-hardness proofs of bilevel optimization; see, e.g.,
\textcite{Marcotte_Savard:2005}.
However, both variants of optimistic bilevel optimization with and
without coupling constraints are NP-complete; see
\textcite{Buchheim:2023}.
Hence, there is no difference between the two variants of optimistic
bilevel optimization in terms of their computational complexity.
Moreover, we showed in a previous paper \parencite{Henke-et-al:2024}
that---although they differ with respect to modeling different types
of feasible sets---they do not differ on the level of optimal
solutions.
More specifically, for every optimistic bilevel problem with coupling
constraints, one can derive another optimistic bilevel problem without
coupling constraints that has the same set of globally optimal
solutions.

In pessimistic bilevel optimization, the literature on problems with
coupling constraints is rather scarce; see, e.g.,
\textcite{Wiesemann-at-al:2013,Zeng:2020} or
the recent survey by \textcite{Beck_et_al:2023a}.
Particularly, some of the main theoretical contributions such as
\textcite{Dempe-et-al:2014,Aussel2019} only consider the case without
coupling constraints.
Somehow, the common sense seemed to be that pessimistic problems with
coupling constraints are much harder to deal with than their variants
without coupling constraints.
The main contribution of this note is to show that this is not the
case for linear bilevel problems.
To be more precise, we use the techniques introduced in
\textcite{Zeng:2020,Henke-et-al:2024} to show that, for every
pessimistic bilevel optimization problem with coupling constraints, we
can state a pessimistic bilevel optimization problem without coupling
constraints that has the same set of globally optimal solutions.
Moreover, we even show that, for every pessimistic bilevel optimization
problem with coupling constraints, we can also derive an optimistic
bilevel optimization problem without coupling constraints that has the
same set of globally optimal solutions.

%% file: problem-statement.tex
\section{Problem Statement}
\label{sec:problem-statement}

In this note, we consider different types of linear bilevel
optimization problems.
The most basic one is the so-called optimistic bilevel optimization
problem without coupling constraints, which reads
\begin{equation}
  \label{eq:bilevel-opt-wo-ccs}
  \min_{x \in X} \quad
  F_{\text{o}}(x) \define c^\top x
  + \min_y \Defset{d^\top y}{y \in S(x)},
\end{equation}
where $S(x)$ is the set of optimal solutions to the $x$-parameterized
optimization problem
\begin{equation}
  \label{eq:bilevel-opt-ll}
  \min_y \quad f^\top y
  \quad \st \quad
  Cx + Dy \geq b.
\end{equation}
Here and in what follows, all variables are continuous and we omit the
dimensions for better readability.
Problem~\eqref{eq:bilevel-opt-wo-ccs} is called optimistic because the leader
is able to choose the lower-level variable~$y$ among all optimal ones
for the follower's problem if there are any multiplicities.

The extended version of this optimistic problem
that includes coupling constraints,
i.e., upper-level constraints depending on the lower-level variables,
can be written as
\begin{equation}
  \label{eq:bilevel-opt}
    \min_{x \in X} \quad
    F_{\text{oc}}(x) \define c^\top x
    + \min_y \Defset{d^\top y}{y \in S(x), \, Ax + By \geq a},
\end{equation}
where $S(x)$ is defined as before.

There is also the pessimistic bilevel problem, which
we again study in two different versions.
In the first one, only the upper-level objective function but not the
upper-level constraints depend on the follower's variables.
This problem is given by
\begin{equation}
  \label{eq:bilevel-press-wo-ccs}
  \min_{x \in \bar{X}} \quad
  F_{\text{p}}(x) \define c^\top x + \max_y \Defset{d^\top y}{y \in S(x)} \quad
\end{equation}
with
\begin{equation}
  \label{eq:bilevel-press-wo-ccs:new-ul-set}
  \bar{X} \define X \cap \Defset{x}{S(x) \neq \emptyset}.
\end{equation}
Here, $S(x)$ is again the same as before.

In the second version, the upper-level constraints depend on~$y$,
but we assume that the upper-level objective function does not depend
on~$y$ anymore.
This assumption can be made w.l.o.g.\ by using the classic epigraph
reformulation \rev{\parencite{Stein:2025}}.
The problem then reads
\begin{subequations}
  \label{eq:bilevel-press-w-ccs}
  \begin{align}
    \min_{x \in \bar{X}} \quad
    & F_{\text{pc}}(x) \define c^\top x
    \\
    \st \quad
    & Ax + By \geq a \quad \text{for all } y \in S(x).
    \label{eq:bilevel-press-w-ccs:cc}
  \end{align}
\end{subequations}

Let us note here that the pessimistic bilevel problems are stated
above in a slightly different way compared to what one usually finds in the
literature.
The difference is that we ensure the non-emptiness of the
lower level's solution set
in~\eqref{eq:bilevel-press-wo-ccs:new-ul-set}.
The problem is that, without this constraint, for the case without
coupling constraints, it would be the best for the leader to choose
an~$x \in X$ for which~$S(x) = \emptyset$.
The reason is that the inner optimization problem would then be a
maximization over the empty set, which formally evaluates to $-
\infty$ and which is the best possible outcome for the outer
minimization problem.
However, in many (even pessimistic) situations, the leader does not
want to actually make the follower's problem infeasible.
For instance, \textcite{Dempe-et-al:2014,Aussel2019} consider this
pessimistic setting without coupling constraints and without the
constraint~$S(x) \neq \emptyset$.
However, they make the assumption that $S(x) \neq \emptyset$ is satisfied
for every
$x\in X$, which resolves the situation sketched above.
On the other hand, the pessimistic problem with coupling constraints
is considered in, e.g.,
\textcite{Tahernejad-et-al:2020,Wiesemann-at-al:2013}.
However, both papers do not state any respective constraint or
assumption on $S(x)$.
This again means that it is possible that the best leader's strategy
would be to choose an~$x \in X$ so that $S(x) = \emptyset$ holds.
This implies that Constraint~\eqref{eq:bilevel-press-w-ccs:cc} vanishes.
Although this does not need to be wrong in a formal sense, in our
opinion, it is at least questionable if this is what really
should be modeled.
\rev{Note that the assumption $S(x) \neq \emptyset$ is not required for
  the optimistic problem~\eqref{eq:bilevel-opt} since choosing~$x$
  so that $S(x) = \emptyset$ is the worse option of the leader
  because, here, the inner and then also the outer minimization
  problem would evaluate to $+\infty$.}

For streamlining the presentation of the core ideas of this paper, we
make the following standing assumptions.
\begin{standingassumption}
  \begin{enumerate}
  \item[(i)] For all $x \in X$, the set $\defset{y}{Cx + Dy \geq b}$
    is non-empty and compact.
  \item[(ii)] The set $X$ is a non-empty polyhedron and all vectors
    and matrices have rational entries.
  \item[(iii)] All upper-level objective functions are bounded from below
    on~$X$.
  \end{enumerate}
\end{standingassumption}
The first assumption is sufficient to guarantee that $S(x) \neq
\emptyset$ holds for all $x \in X$.
The latter is also used in \textcite{Dempe-et-al:2014,Aussel2019}.
The second assumption is required for applying the results from
\textcite{Henke-et-al:2024} and the third one ensures solvability of
the overall problem.

Before we move on to our main results, let us brief\/ly discuss what
we consider to be a solution to the stated bilevel problems.
In the literature, it is not handled in a unique way if either $x$ or
$(x,y)$ is a solution to a bilevel problem.
For sure, for the pessimistic bilevel
problem~\eqref{eq:bilevel-press-w-ccs}, the solution can only be in
the $x$-space since the variables~$y$ of the follower are connected to
a universal quantifier in~\eqref{eq:bilevel-press-w-ccs:cc}.
In order to be consistent, we also consider the solutions to
\eqref{eq:bilevel-opt-wo-ccs}, \eqref{eq:bilevel-opt}, and
\eqref{eq:bilevel-press-wo-ccs} to be in the
$x$-space and interpret a respective~$y$, if given at all, only as a
certificate for the optimality of~$x$.
For a detailed discussion about the representation of solutions of bilevel
problems, we refer to Section~2.6 in \textcite{Goerigk_et_al:2025}.

%% file: removing-ccs.tex
\section{An Equivalent Reformulation without Coupling Constraints}
\label{sec:reform-with-coupl}

We now prove that pessimistic bilevel optimization with and without
coupling constraints are equivalent on the level of globally optimal
solutions.
To this end, we derive multiple reformulations and show their
equivalence properties.
The main proof strategy is given in
Figure~\ref{fig:models-and-reforms}.
The top, right, and bottom arc are considered in the respective
following sections.

\begin{figure}
  \centering
  \input{models-and-reform.tikz}
  \caption{Models and Reformulations}
  \label{fig:models-and-reforms}
\end{figure}
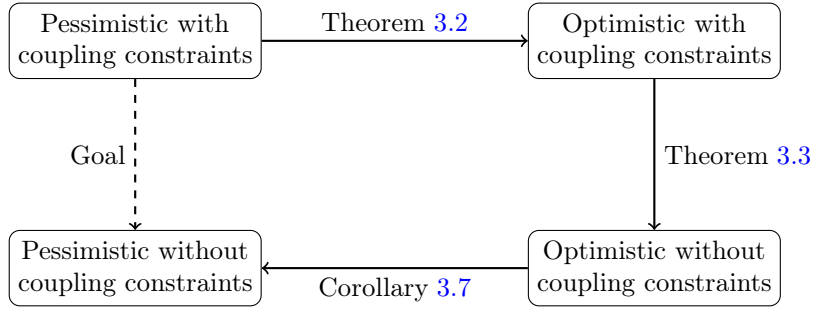

\input{pess-w-cc-to-opt-w-cc}

\input{opt-w-cc-to-opt-wo-cc}

\input{opt-wo-cc-to-pess-wo-cc}

To sum up the overall section, we showed that we can reformulate a
pessimistic bilevel optimization problem with coupling
constraints as one without; see Figure~\ref{fig:models-and-reforms}.
Moreover, note that the resulting problem is of polynomial size in the
input size of the originally given one and that this reformulation can
be done in polynomial time.

%% file: models-and-reform.tikz
\begin{tikzpicture}[node distance=2cm and 3.5cm, auto]
    \node (A) [draw, rectangle, rounded corners, minimum width=2.5cm,
    minimum height=1cm, align=center] {Pessimistic with\\coupling
      constraints};

    \node (B) [draw, rectangle, rounded corners, right=of A,
    minimum width=2.5cm, minimum height=1cm, align=center]
    {Optimistic with\\coupling constraints};

    \node (C) [draw, rectangle, rounded corners, below=of B,
    minimum width=2.5cm, minimum height=1cm, align=center]
    {Optimistic without\\coupling constraints};

    \node (D) [draw, rectangle, rounded corners, below=of A,
    minimum width=2.5cm, minimum height=1cm, align=center]
    {Pessimistic without\\coupling constraints};

    \draw[->, thick] (A) -- node[above] {Theorem~\ref{thm:pwcc-to-owcc}} (B);
    \draw[->, thick] (B) -- node[right] {Theorem~\ref{thm:owcc-to-o-wo-cc}} (C);
    \draw[->, thick] (C) -- node[below] {\rev{Theorem}~\ref{thm:opt-wo-cc-to-pess-wo-cc}} (D);
    \draw[->, thick, dashed] (A) -- node[left] {Goal} (D);
\end{tikzpicture}%

%% file: pess-w-cc-to-opt-w-cc.tex
\subsection{From Pessimistic to Optimistic Bilevel
  Optimization with Coupling Constraints}
\label{sec:pess-w-cc-to-opt-w-cc}

First, we reformulate the pessimistic bilevel optimization problem
with coupling constraints, i.e.,
Problem~\eqref{eq:bilevel-press-w-ccs}, as an optimistic bilevel
problem with coupling constraints.
To this end, we use the results from Section~3.3 in
\textcite{Zeng:2020}, in particular Lemma~3.
This first requires to re-write the coupling constraints
\begin{equation*}
  Ax + By \geq a \quad \text{for all } y \in S(x)
\end{equation*}
in a component-wise way.
For this, let $A \in \R^{m \times {n_x}}$, $B \in \R^{m \times n_y}$,
and $a \in \R^m$ with $m$ being the number of coupling constraints and
with $n_x$ and $n_y$ being the numbers of upper- and lower-level variables,
respectively.
Hence, we have
\begin{equation}
  \label{eq:coupling-constrs-comp-wise}
  A_{i\cdot} x + B_{i\cdot} y \geq a_i
  \quad \text{for all } y \in S(x)
  \text{ and all } i \in [m] \define \Set{1, \dotsc, m}.
\end{equation}

We are now able to re-phrase Lemma~3 by \textcite{Zeng:2020} in our
notation.
For being more self-complete, we also add a simplified proof here.
\begin{lemma}
  \label{lem:bo-zong-aux-lemma}
  Let $x \in X$ be given and consider a fixed~$i \in [m]$.
  Then, $x$ satisfies the $i$-th coupling constraint
  in~\eqref{eq:coupling-constrs-comp-wise} if and only if there exist
  $\bar{y}$ and
  \begin{equation*}
    y^i \in \argmin \Defset{B_{i\cdot} y}{Dy \geq b - Cx, \, f^\top y
      \leq f^\top \bar{y}}
  \end{equation*}
  that satisfy
  \begin{equation*}
    D \bar{y} \geq b - Cx,
    \quad
    B_{i\cdot} y^i \geq a_i - A_{i\cdot} x.
  \end{equation*}
\end{lemma}
\begin{proof}
  Let $x\in X$ be given and consider a fixed $i\in[m]$.
  Then, the $i$-th coupling
  constraint in~\eqref{eq:coupling-constrs-comp-wise} is
  equivalent to $ \min_{y}\defset{ B_{i\cdot}y}{ y\in S(x) } \ge a_i -
  A_{i\cdot}x$, which can be reformulated as $ B_{i\cdot}y^i \ge a_i -
  A_{i\cdot}x $ with
  $
    y^i \in \argmin_y \Defset{ B_{i\cdot}y }{ y\in S(x) }.
  $
  Now, let $\varphi$ denote the optimal-value function of the lower-level
  problem~\eqref{eq:bilevel-opt-ll}. It follows that $x$ satisfies the $i$-th
  coupling constraint if and only if
  \begin{equation}
    \label{value-function-bo-zeng}
    B_{i\cdot}y^i \ge a_i - A_{i\cdot}x \quad\text{with}\quad y^i \in
    \argmin_y\Defset{ B_{i\cdot}y }{ Dy \ge b - Cx, \, f^\top y \le \varphi(x) }.
  \end{equation}

  We now show that the latter is equivalent to the stated conditions in the
  lemma. First, let us assume that~\eqref{value-function-bo-zeng} holds. Then,
  there exists $\bar{y}$ such that $\varphi(x) = f^\top \bar{y}$ and $D\bar y
  \ge b - Cx$ is satisfied. Hence, the conditions of the lemma hold.

  Conversely, assume that the conditions of the lemma are satisfied.
  The feasibility of $\bar{y}$ implies
  \begin{equation}\label{aux-statement-bo-zeng}
    \min_y\Defset{ B_{i\cdot}y }{ Dy \ge b - Cx, \, f^\top y \le f^\top \bar y } \le
    \min_y\Defset{ B_{i\cdot}y }{ Dy \ge b - Cx, \, f^\top y \le \varphi(x) }.
  \end{equation}
  Hence, \eqref{value-function-bo-zeng} is satisfied, which concludes the proof.
\end{proof}

\begin{corollary}
  \label{cor:bo-zong-aux-lemma}
  \rev{
  A given $x \in X$ satisfies all coupling constraints
  in~\eqref{eq:coupling-constrs-comp-wise} if and only if there exist
  $\bar{y}$ satisfying $D \bar{y} \geq b - Cx$ and, for every~$i \in [m]$,
  \begin{equation*}
    y^i \in \argmin \Defset{B_{i\cdot} y}{Dy \geq b - Cx, \, f^\top y
      \leq f^\top \bar{y}}
  \end{equation*}
  with}
  \begin{equation*}
    \rev{B_{i\cdot} y^i \geq a_i - A_{i\cdot} x.}
  \end{equation*}
\end{corollary}
\begin{proof}
  \rev{
  In line with the proof of Lemma~\ref{lem:bo-zong-aux-lemma}, it follows that a
  given leader's decision~$x \in X$ satisfies all coupling constraints
  in~\eqref{eq:coupling-constrs-comp-wise} if and only if, for each~$i \in [m]$,
  Condition~\eqref{value-function-bo-zeng} is satisfied.
  Consequently, if the latter conditions hold, there exists $\bar{y}$ such that
  $\varphi(x) = f^\top \bar{y}$
  and $D\bar y \ge b - Cx$.
  This implies the conditions of the corollary.
  Conversely, assuming that the conditions of the corollary hold, the claim
  follows using Inequality~\eqref{aux-statement-bo-zeng} analogously to the
  proof of Lemma~\ref{lem:bo-zong-aux-lemma}.
  }
\end{proof}

From this \rev{corollary}, we obtain the following result.

\begin{theorem}
  \label{thm:pwcc-to-owcc}
  Let $\mathcal{S}$ be the set of globally optimal solutions to the
  pessimistic bilevel problem~\eqref{eq:bilevel-press-w-ccs} with
  coupling constraints, i.e., to the problem
  \begin{subequations}
    \label{eq:bo-zeng-orig-problem}
    \begin{align}
      \min_{x \in X} \quad
      & F_{\mathrm{pc}}(x) = c^\top x
      \\
      \st \quad
      & Ax + By \geq a \quad \text{for all } y \in S(x).
    \end{align}
  \end{subequations}
  Moreover, let $\hat{\mathcal{S}}$ be the set of globally
  optimal solutions to the optimistic bilevel problem
  \begin{equation*}
    \min_{(x, \bar{y}) \in \tilde{X}} \
    c^\top x + \min_y \Defset{0}{y \in \hat{S}(x, \bar{y}), \,
      B_{i\cdot} y^i \geq a_i - A_{i\cdot} x \text{ for all } i \in [m]},
  \end{equation*}
  with $\tilde{X} = \defset{(x, \bar{y})}{x \in X, D \bar{y} \geq b -
    Cx}$ and with $\hat{S}(x, \bar{y})$ being the set of optimal
  solutions to the lower-level problem
  \begin{align*}
    \min_y \quad
    & \sum_{i=1}^m B_{i\cdot} y^i
    \\
    \st \quad
    & Dy^i \geq b - Cx
      \quad \text{for all } i \in [m],
    \\
    & f^\top y^i \leq f^\top \bar{y}
      \quad \text{for all } i \in [m].
  \end{align*}
  Then,
  \begin{equation*}
    \mathcal{S} = \proj_x(\hat{\mathcal{S}})
  \end{equation*}
  holds and the optimal objective function values coincide.
\end{theorem}
\begin{proof}
  \rev{We prove the theorem with the help of an auxiliary and intermediate
  problem.
  To this end, let $\tilde{\mathcal{S}}$ be the set of globally optimal
  solutions to the optimistic single-leader multi-follower problem
  \begin{equation}
    \label{eq:bo-zeng-multi-follower}
    \min_{(x, \bar{y}) \in \tilde{X}} \
    c^\top x + \min_{y} \Defset{0}{y^i \in \tilde{S}^i(x,\bar{y}), \,
      B_{i\cdot} y^i \geq a_i - A_{i\cdot} x \text{ for all } i \in [m]}
  \end{equation}
  with $\tilde{S}^i(x,\bar{y}) = \argmin_{y'} \defset{B_{i\cdot}
    y'}{Dy' \geq b - Cx, f^\top y' \leq f^\top \bar{y}}$
  and $y = (y^i)_{i=1}^m$.}
  The identity $\proj_x(\tilde{\mathcal{S}}) =
  \proj_x(\hat{\mathcal{S}})$ is straightforward because all
  follower's problems in~\eqref{eq:bo-zeng-multi-follower} are
  independent.
  \rev{Hence, the proof is complete if we show $\mathcal{S} =
  \proj_x(\tilde{\mathcal{S}})$.}
  We start by proving $\mathcal{S} \subseteq
  \proj_x(\tilde{\mathcal{S}})$.
  To this end, let~$x$ be feasible for
  Problem~\eqref{eq:bo-zeng-orig-problem}, implying $x \in X$.
  \rev{
  By applying \rev{Corollary~\ref{cor:bo-zong-aux-lemma}},
  there exist~$\bar{y}$ with~$D \bar{y}
  \geq b - Cx$ and, for every~$i \in [m]$, some $y^i \in \argmin_y
  \defset{B_{i\cdot} y}{Dy \geq b - Cx, f^\top y \leq f^\top \bar{y}}$
  satisfying  $B_{i\cdot} y^i \geq a_i - A_{i\cdot} x$.}
  Thus, $y^i \in \tilde{S}^i(x,\bar{y})$ holds, meaning that, for the
  given~$x$, there exists $\bar{y}$ so that $(x,\bar{y})$ is feasible
  for \eqref{eq:bo-zeng-multi-follower}.
  Since the respective upper-level objective functions coincide and
  only depend on~$x$, we showed $\mathcal{S} \subseteq
  \proj_x(\tilde{\mathcal{S}})$.
  Due to Lemma~\ref{lem:bo-zong-aux-lemma} being an if-and-only-if
  statement, the other direction follows by using the same
  arguments.
\end{proof}

Hence, we have shown that a pessimistic bilevel problem with coupling
constraints can be equivalently re-written as an optimistic bilevel
problem with coupling constraints.

%% file: opt-w-cc-to-opt-wo-cc.tex
\subsection{From Optimistic Bilevel Optimization with to without
  Coupling Constraints}
\label{sec:opt-w-to-wo-ccs}

To reformulate an optimistic bilevel optimization problem
with coupling constraints as an optimistic bilevel optimization
problem without coupling constraints,
we use the main result of~\textcite{Henke-et-al:2024}.
\rev{We restate this result here so that it fits directly to our
  setup.}

\begin{theorem}[Corollary 2.3 in~\textcite{Henke-et-al:2024}]
  \label{thm:owcc-to-o-wo-cc}
  There is a polynomial-sized (in the bit-encoding length of the
  problem's data) penalty parameter~$\kappa > 0$ so
  that the optimistic bilevel problem~\eqref{eq:bilevel-opt} with
  coupling constraints has the same set of globally optimal solutions
  as the optimistic bilevel problem
  \begin{equation*}
    \min_{x \in X} \quad c^\top x
    + \min_{y, \varepsilon}
    \Defset{d^\top y + \kappa \varepsilon}{(y, \varepsilon) \in
      S'(x)}
  \end{equation*}
  without coupling constraints, where $S'(x)$ is the set of optimal
  solutions to the $x$-parameterized lower-level problem
  \begin{align*}
    \min_{y, \varepsilon} \quad
    & f^\top y \\
    \st \quad
    & Ax+By + \varepsilon e \geq a, \\
    & Cx + Dy \geq b, \\
    & \varepsilon \geq 0,
  \end{align*}
  where $e$ is the vector of all ones in appropriate dimension.
  Moreover, both bilevel problems have the same optimal objective
  function value.
\end{theorem}

\textcite{Henke-et-al:2024} state as an open question how to compute
the penalty parameter~$\kappa$ in polynomial time.
This question is answered by Lemma~4 of
\textcite{Lefebvre-Schmidt:2024}.

%% file: opt-wo-cc-to-pess-wo-cc.tex
\subsection{From Optimistic to Pessimistic Bilevel Optimization
  without Coupling Constraints}
\label{sec:opt-wo-cc-to-pess-wo-cc}

In this section, we show how to reformulate an optimistic bilevel
problem~\eqref{eq:bilevel-opt-wo-ccs} without coupling constraints as
a pessimistic bilevel problem without coupling constraints so that the
globally optimal solutions coincide.
To this end, we consider the following auxiliary optimistic bilevel
problem
\begin{equation}
  \label{eq:bilevel-epsilon}
  \min_{(x,\bar{y})\in \tilde{X}} \quad F_{\text{oa}}(x,\bar{y})
  \define c^\top x + d^\top \bar{y} + \min_{y,\varepsilon} \Defset{0}{
    \varepsilon = 0, \, (y,\varepsilon) \in \tilde S(x, \bar{y}) }
\end{equation}
with a single coupling constraint.
Again, we use $\tilde{X} = \Defset{(x, \bar{y})}{x \in X, D
  \bar{y} \ge b - Cx }$, and $\tilde S(x, \bar{y})$ denotes the set of
optimal points to
\begin{subequations}
  \label{eq:bilevel-epsilon-lower}
  \begin{align}
    \min_{y,\varepsilon} \quad
    & f^\top y \\
    \text{s.t.} \quad
    & Cx + Dy \ge b, \\
    & f^\top \bar{y} - f^\top y  = \varepsilon,
      \label{eq:eps-lower-level-objective-value} \\
    & \varepsilon \ge 0.
  \end{align}
\end{subequations}
The main intuition behind this problem is that the leader can choose
the most favorable lower-level feasible point~$\bar{y}$ in terms of
her objective function while the follower computes the non-negative
difference between the actual optimal lower-level objective value and
the one corresponding to~$\bar{y}$ in
Constraint~\eqref{eq:eps-lower-level-objective-value}.
Finally, the coupling constraint~$\varepsilon=0$ ensures that this
difference is zero, i.e., the leader's decision~$\bar{y}$ is also
optimal for the lower-level problem~\eqref{eq:bilevel-opt-ll}.

\begin{lemma} \label{lemma:opt-wo-cc-to-bilevel-epsilon}
  For every bilevel feasible point $x$ of the optimistic bilevel
  problem~\eqref{eq:bilevel-opt-wo-ccs} without coupling constraints,
  the point $(x,\bar{y})$ with $\bar{y}\in \argmin_{y} \defset{ d^\top
    y}{ y \in S(x) }$
  is also bilevel feasible for the optimistic bilevel
  problem~\eqref{eq:bilevel-epsilon} with the same objective value.
  Moreover, for every globally optimal point $(x,\bar{y})$ to
  Problem~\eqref{eq:bilevel-epsilon}, $x$ is bilevel feasible
  for~\eqref{eq:bilevel-opt-wo-ccs} with the same objective value.
\end{lemma}
\begin{proof}
  Because $x$ is bilevel feasible
  for~\eqref{eq:bilevel-opt-wo-ccs}, there exists a point~$y\in S(x)$
  such that $F_{\text{o}}(x) = c^\top x + d^\top y$.
  Note that the objective function
  of~\eqref{eq:bilevel-epsilon-lower} does not depend on
  $\varepsilon$.
  Moreover, for fixed $x$ and any $\bar{y}$ satisfying~$Cx +
  D\bar{y} \ge b$, the inequality $f^\top\bar{y} \ge f^\top y$ holds.
  Consequently, the optimal objective value
  of~\eqref{eq:bilevel-epsilon-lower} is exactly $f^\top y$.
  Thus, for a given $x$ and $\bar{y} \define y$, we have $(y,0)\in \tilde
  S(x,\bar{y})$ and $(x,\bar{y})$ is a
  bilevel feasible point for~\eqref{eq:bilevel-epsilon} with
  $F_{\text{oa}}(x,\bar{y}) = F_{\text{o}}(x)$.

  Conversely, because $(x,\bar{y})$ is a globally optimal point
  for~\eqref{eq:bilevel-epsilon}, there exists $(y,\varepsilon)\in
  \tilde{S}(x,\bar{y})$ with $\varepsilon = 0$.
  Consequently, $f^{\top} y = f^{\top} \bar{y}$ holds.
  Following the same line of arguments as above, we obtain $y\in
  S(x)$.
  This implies that~$x$ is feasible for~\eqref{eq:bilevel-opt-wo-ccs}.
  We are left to prove that $F_{\text{oa}}(x,\bar{y}) =
  F_{\text{o}}(x)$ holds.
  Assume that this is not the case.
  Hence, $d^\top y \neq d^\top \bar{y}$ needs to hold.
  If $d^\top y < d^\top \bar{y}$, we could choose $(x, \hat{y})$ with
  $\hat{y} \define y$.
  This yields $(y,0) \in \tilde{S}(x,\hat{y})$, which
  again implies $F_{\text{oa}}(x,\hat{y}) = F_{\text{o}}(x)$.
  This contradicts the optimality of $(x, \bar{y}$).
  If $d^\top y > d^\top \bar{y}$, then $\bar{y}$ would be a
  better lower-level solution than~$y$ in terms of the leader's
  objective function in \eqref{eq:bilevel-opt-wo-ccs}, which, again,
  is a contradiction.
  Hence, $F_{\text{oa}}(x,\bar{y}) = F_{\text{o}}(x)$ holds, which
  ends the proof.
\end{proof}

Using the same proof techniques that lead to Theorem~2.2
of~\textcite{Henke-et-al:2024}, we move the single coupling
constraint~$\varepsilon=0$ of \eqref{eq:bilevel-epsilon} to the
leader's objective function.
This yields the following lemma.

\begin{lemma}
  \label{lemma:henke-penalized-eps-version}
  There is a polynomial-sized parameter $\kappa > 0$ so that
  Problem~\eqref{eq:bilevel-epsilon} has the same set of globally
  optimal solutions as the optimistic bilevel problem
  \begin{equation}
    \label{eq:bilevel-epsilon-wo-coupling}
    \min_{(x,\bar{y})\in \tilde{X}} \quad F_{\mathrm{o\kappa}}(x,\bar{y})
    \define c^\top x + d^\top \bar{y} + \min_{y,\varepsilon}
    \Defset{\kappa \varepsilon}{(y,\varepsilon) \in \tilde S(x, \bar{y}) }
  \end{equation}
  without coupling constraints.
  Here, we again use $\tilde{X} = \defset{(x, \bar{y})}{x \in X, D
    \bar{y} \ge b - Cx }$ and $\tilde S(x, \bar{y})$ is the set of
  optimal solutions of~\eqref{eq:bilevel-epsilon-lower}.
\end{lemma}

\begin{proposition}
  \label{prop:opt-wo-cc-to-pess-wo-cc-for-epsilon-problem}
  For any~$\kappa$, the optimistic bilevel
  problem~\eqref{eq:bilevel-epsilon-wo-coupling} without
  coupling constraints and its pessimistic version
  \begin{equation}
    \label{eq:bilevel-epsilon-wo-coupling-pessimistic}
    \min_{(x,\bar{y})\in \tilde{X}} \quad F_{\mathrm{p\kappa}}(x,\bar{y})
    \define c^\top x + d^\top \bar{y}
    + \max_{y,\varepsilon} \Defset{\kappa \varepsilon}{
      (y,\varepsilon) \in \tilde S(x, \bar{y}) }
  \end{equation}
  have the same set of feasible and globally optimal solutions.
\end{proposition}
\begin{proof}
  Both problems have the same feasible sets.
  Moreover, for any feasible point $(x,\bar{y})$,
  the value of $\varepsilon$ is uniquely determined.
  Thus, the inner minimization problem in
  \eqref{eq:bilevel-epsilon-wo-coupling} and the inner maximization
  problem in \eqref{eq:bilevel-epsilon-wo-coupling-pessimistic} have
  the same value.
  Thus, $F_{\text{o}\kappa}(x, \bar{y}) = F_{\text{p}\kappa}(x,
  \bar{y})$ holds.
\end{proof}

\begin{theorem}
  \label{thm:opt-wo-cc-to-pess-wo-cc}
  There is a polynomial-sized parameter $\kappa > 0$ so that
  the optimistic bilevel problem~\eqref{eq:bilevel-opt-wo-ccs} without
  coupling constraints has the same set of globally optimal solutions
  as the pessimistic bilevel
  problem~\eqref{eq:bilevel-epsilon-wo-coupling-pessimistic} without
  coupling constraints.
\end{theorem}
\begin{proof}
  The claim follows from
  Lemmas~\ref{lemma:opt-wo-cc-to-bilevel-epsilon}
  and~\ref{lemma:henke-penalized-eps-version} as well as
  \rev{Proposition~\ref{prop:opt-wo-cc-to-pess-wo-cc-for-epsilon-problem}}.
\end{proof}

%% file: conclusion.tex
\section{Conclusion}
\label{sec:conclusion}

We show in this note that---on the level of globally optimal
solutions---there is no difference between linear pessimistic bilevel
optimization with and without coupling constraints.
To be more precise, for a given pessimistic bilevel
optimization problem with coupling constraints, we can derive another one
without coupling constraints having the same global optimizers.
Moreover, we even show that we can go from a pessimistic problem with
coupling constraints to an optimistic problem without coupling
constraints---again having the same global solutions.

It was somehow common sense that having coupling constraints or not
makes a significant difference in pessimistic bilevel optimization.
It is now shown that this is not the case.
In particular, many novel theoretical results or even solution
techniques can be gathered for pessimistic problems with coupling
constraints by simply studying an equivalent problem without such
constraints---or even an optimistic problem.

%% file: acknowledgement.tex
\section*{Acknowledgements}

The second and third author acknowledge the support by the German
Bundesministerium für Bildung und Forschung within
the project \enquote{RODES} (Förderkennzeichen 05M22UTB).